\documentclass[a4paper]{amsart}
\usepackage{txfonts, amsmath,amstext,amsthm,amscd,amsopn,verbatim,amssymb, amsfonts}
\usepackage{fullpage}

\usepackage{float}
\restylefloat{table}

\usepackage{tikz}
\usepackage{tikz-cd}

\usepackage[all,cmtip]{xy}

\usetikzlibrary{matrix}
\usetikzlibrary{shapes}
\usetikzlibrary{arrows}
\usetikzlibrary{calc,3d}
\usetikzlibrary{decorations,decorations.pathmorphing,decorations.pathreplacing,decorations.markings}
\usetikzlibrary{through}

\tikzset{ext/.style={circle, draw,inner sep=1pt},int/.style={circle,draw,fill,inner sep=1.4pt},nil/.style={inner sep=1pt}}
\tikzset{cy/.style={circle,draw,fill,inner sep=2pt},scy/.style={circle,draw,inner sep=2pt},scyx/.style={draw,cross out,inner sep=2pt},scyt/.style={draw,regular polygon,regular polygon sides=3,inner sep=0.95pt}}
\tikzset{exte/.style={circle, draw,inner sep=3pt},inte/.style={circle,draw,fill,inner sep=3pt}}
\tikzset{diagram/.style={matrix of math nodes, row sep=3em, column sep=2.5em, text height=1.5ex, text depth=0.25ex}}
\tikzset{diagram2/.style={matrix of math nodes, row sep=0.5em, column sep=0.5em, text height=1.5ex, text depth=0.25ex}}

\tikzset{
  rightblue/.style={
    decoration={markings,mark=at position .8 with {\arrow[scale=1.2,blue]{latex}}},
    postaction={decorate},begin
    shorten >=0.4pt}}
\tikzset{
  leftblue/.style={
    decoration={markings,mark=at position .6 with {\arrowreversed[scale=1.2,blue]{latex}}},
    postaction={decorate},
    shorten >=0.4pt}}
\tikzset{
  rightred/.style={
    decoration={markings,mark=at position .4 with {\arrow[scale=1.2,red]{latex}}},
    postaction={decorate},
    shorten >=0.4pt}}
\tikzset{
  leftred/.style={
    decoration={markings,mark=at position .2 with {\arrowreversed[scale=1.2,red]{latex}}},
    postaction={decorate},
    shorten >=0.4pt}}
    
\tikzset{
  crossed/.style={
    decoration={markings,mark=at position .5 with {\arrow{|}}},
    postaction={decorate},
    shorten >=0.4pt}}

\newcommand{\tto}[1]{{\overset{#1}{\longrightarrow}}}

\usepackage{chngcntr}
\counterwithin{figure}{section}
\theoremstyle{plain}
  \newtheorem{thm}[figure]{Theorem}
  
  \newtheorem{prop}[figure]{Proposition}
  
  \newtheorem{cor}[figure]{Corollary}

\theoremstyle{definition}
  
  \newtheorem{rem}[figure]{Remark}

\newcommand{\K}{{\mathbb{K}}}

\newcommand{\sym}{{\mathbb{S}}}

\newcommand{\GC}{\mathrm{GC}}

\newcommand{\HGC}{\mathrm{HGC}}

\newcommand{\fHGCc}{\mathrm{fHGCc}}

\newcommand{\coH}{{\mathcal{H}}}

\newcommand{\mV}{\mathrm{V}}
\newcommand{\mE}{\mathrm{E}}
\newcommand{\mH}{\mathrm{H}}
\newcommand{\mB}{\mathrm{B}}

\DeclareMathOperator{\sgn}{sgn}

\newcommand{\grac}{\mathrm{grac}}

\begin{document}
\title{Second extra differential on odd graph complexes}

\author{Marko \v Zivkovi\' c}
\address{Mathematics Research Unit, University of Luxembourg\\ 
Maison du Nombre, 6, avenue de la Fonte, L-4364 Esch-sur-Alzette\\
Grand Duchy of Luxembourg}



\begin{abstract}
A suitable extra differential on graph complexes can lead to a pairing of its cohomological classes. Many such extra differentials are known for various graph complexes, including Kontsevich's graph complex $\GC_n$ for odd $n$. In this paper we introduce another extra differential on the same graph complex, leading to another way of pairing of its cohomological classes. Two ways of pairing lead to even further understanding of graph cohomology through ``waterfall mechanism''.
\end{abstract}

\maketitle

\section{Introduction}

Generally speaking, graph complexes are graded vector spaces of formal linear combinations of isomorphism classes of some kind of graphs, with the standard differential defined by vertex splitting (or, dually, edge contraction). The various graph cohomology theories are arguably some of the most fascinating objects in homological algebra. Each of graph complexes play a certain role in a subfield of homological algebra, algebraic topology or mathematical physics. They have an elementary and simple combinatorial definition, yet we know little about what their cohomology actually is.

The simplest graph complex is introduced by Maxim Kontsevich in \cite{Kont1,Kont2}. It comes in versions $\GC_n$, where $n$ ranges over integers. The differential $\delta$ splits a vertex (see Section \ref{s:def} for the definition). Physically, $\GC_n$ is formed by vacuum Feynman diagrams of a topological field theory in dimension $n$. Alternatively, $\GC_n$ governs the deformation theory of $E_n$ operads in algebraic topology \cite{grt} or stable cohomology of the algebraic polyvector fields \cite{Poly}. Some examples of graphs are:
$$
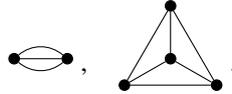
\begin{tikzpicture}[baseline=-1ex,scale=.7]
\node[int] (a) at (0,0) {};
\node[int] (b) at (1,0) {};
\draw (a) edge (b);
\draw (a) edge[bend left=40] (b);
\draw (a) edge[bend right=40] (b);
\end{tikzpicture}
\;,\quad
\begin{tikzpicture}[baseline=-1ex,scale=.7]
\node[int] (a) at (0,0) {};
\node[int] (b) at (90:1) {};
\node[int] (c) at (210:1) {};
\node[int] (d) at (330:1) {};
\draw (a) edge (b);
\draw (a) edge (c);
\draw (a) edge (d);
\draw (b) edge (c);
\draw (b) edge (d);
\draw (c) edge (d);
\end{tikzpicture}
\;.
$$

Kontsevich's graph complexes split into the product of sub-complexes with fixed loop order $b$:
$$
\GC_n=\prod_{b\geq -1}\mB_b\GC_n.
$$
Furthermore, the complexes $\GC_n$ and $\GC_{n'}$ are isomorphic up to some unimportant degree shifts if $m\equiv m' \text{ mod } 2$, (compare Tables \ref{tbl:GC3} and \ref{tbl:GC1} for an example). Hence it suffices to understand 2 possible cases of $\GC_n$: \emph{even graph complex} for even $n$, and \emph{odd graph complex} for odd $n$.

The long standing open problem we are attacking in this paper is the following.

\smallskip

{\bf Open Problem:} Compute the cohomology $\coH\left(\GC_n\right)$.

\smallskip

In this paper we consider only odd graph complexes. Dimensions of $\coH\left(\GC_3\right)$ are given in Table \ref{tbl:GC3}. As depicted in the table, there are areas where cohomology is known. However, there is a huge area where we still do not know anything.
A notable result about $\coH(\GC_{3})$ is known for its degree $-3$. Graphs that span $\coH^{-3}(\GC_{3})$ form the commutative algebra of 3-graphs \cite{DKC}, c.f.\ \cite{vogel}.
These are graphs whose all vertices are 3-valent, and the commutative product between two graphs deletes one vertex from each graph and connects loose edges from one graph to loose edges from another graph in all possible ways. It turns out that the choice of vertices does not change the result.

\begin{table}[h]
$$
\begin{tikzpicture}[every node/.style={align=center,anchor=base,text width=3ex,nodes={fill=red!30}}]
\matrix (mag) [matrix of nodes,ampersand replacement=\&,row 1/.style={nodes={fill=white}},row 3/.style={nodes={fill=green!30}},column 1/.style={nodes={fill=white}}]
{
{} \&-12 \&-11 \&-10 \&-9 \&-8 \&-7 \&-6 \&-5 \&-4 \&-3 \&-2 \&-1 \& 0 \& 1 \& 2 \& 3 \& 4 \& 5\\
-1\&  \&  \&  \&  \&  \&  \&  \&  \&  \&  \&  \&  \& 0 \& 0 \& 0 \& 0 \& 0 \& 0 \\
|[fill=white]| 0 \&  \&  \&  \&  \&  \&  \&  \&   \&  \&  \& 0 \& 0 \& 1 \& 0 \& 0 \& 0 \& 1 \& 0 \\
1 \&  \&  \&  \&  \&  \&  \&  \&  \& 0 \& |[fill=green!30]| 1 \& 0 \& 0 \& 0 \& 0 \& 0 \& 0 \& 0 \& 0 \\
2 \&  \&  \& \&  \&  \&  \& 0 \& 0 \& 0 \& |[fill=green!30]| 1 \& 0 \& 0 \& 0 \& 0 \& 0 \& 0 \& 0 \& 0\\
3 \& \&  \&  \&  \& 0 \& 0 \& 0 \& 0 \& |[fill=white]| 0 \& |[fill=green!30]| 1 \& 0 \& 0 \& 0 \& 0 \& 0 \& 0 \& 0 \& 0 \\
4 \&  \&  \& 0 \& 0 \& 0 \& 0 \& 0 \& |[fill=white]| 0 \& |[fill=white]| 0 \& |[fill=green!30]| 2 \& 0 \& 0 \& 0 \& 0 \& 0 \& 0 \& 0 \& 0 \\
5 \& 0 \& 0 \& 0 \& 0 \& 0 \& 0 \& |[fill=white]| 1 \& |[fill=white]| 0 \& |[fill=white]| 0 \& |[fill=green!30]| 2 \& 0 \& 0 \& 0 \& 0 \& 0 \& 0 \& 0 \& 0 \\
6 \& 0 \& 0 \& 0 \& 0 \& 0 \& |[fill=white]| 0 \& |[fill=white]| 1 \& |[fill=white]| 0 \& |[fill=white]| 0 \& |[fill=green!30]| 3 \& 0 \& 0 \& 0 \& 0 \& 0 \& 0 \& 0 \& 0 \\
7 \& 0 \& 0 \& 0 \& 0 \& |[fill=white]| 0 \& |[fill=white]| 0 \& |[fill=white]| 2 \& |[fill=white]| 0 \& |[fill=white]| 0 \& |[fill=green!30]| 4 \& 0 \& 0 \& 0 \& 0 \& 0 \& 0 \& 0 \& 0 \\
8 \& 0 \& 0 \& 0 \& |[fill=white]| 0 \& |[fill=white]| ? \& |[fill=white]| ? \& |[fill=white]| ? \& |[fill=white]| ? \& |[fill=white]| ? \& |[fill=green!30]| 5 \& 0 \& 0 \& 0 \& 0 \& 0 \& 0 \& 0 \& 0 \\
9 \& 0 \& 0 \& |[fill=white]| ? \& |[fill=white]| ? \& |[fill=white]| ? \& |[fill=white]| ? \& |[fill=white]| ? \& |[fill=white]| ? \& |[fill=white]| ? \& |[fill=green!30]| 6 \& 0 \& 0 \& 0 \& 0 \& 0 \& 0 \& 0 \& 0 \\
10 \& 0 \& |[fill=white]| ? \& |[fill=white]| ? \& |[fill=white]| ? \& |[fill=white]| ? \& |[fill=white]| ? \& |[fill=white]| ? \& |[fill=white]| ? \& |[fill=white]| ? \& |[fill=green!30]| 8 \& 0 \& 0 \& 0 \& 0 \& 0 \& 0 \& 0 \& 0 \\
11\& |[fill=white]| ? \& |[fill=white]| ? \& |[fill=white]| ? \& |[fill=white]| ? \& |[fill=white]| ? \& |[fill=white]| ? \& |[fill=white]| ? \& |[fill=white]| ? \& |[fill=white]| ? \& |[fill=green!30]| 9 \& 0 \& 0 \& 0 \& 0 \& 0 \& 0 \& 0 \& 0 \\
};
\draw (mag-2-1.north west) -- (mag-1-19.south east);
\draw (mag-1-2.north west) -- (mag-14-1.south east);
\end{tikzpicture}
$$
\caption{\label{tbl:GC3} The table of dimensions of $\coH(\GC_{3})$. The columns represent the degree, and rows represent the loop order $b:=e-v$. All dimensions in this table are known by computer calculations of Bar Natan and McKay \cite{barnatanmk}, and Willwacher.
Empty white area represents cases where there are no graphs.
Red area has zero cohomology because of various theoretical results (c.f.\ \cite{grt}, \cite[Lemma 1.4]{Wrange}).
Horizontal green line in loop order $0$ corresponds to loop classes (c.f.\ \cite{grt}).
Vertical green line in degree $-3$ forms an commutative algebra of 3-graphs \cite{DKC}.}
\end{table}

Cohomology of another odd graph complex $\coH\left(\GC_1\right)$ is obtained from $\coH\left(\GC_3\right)$ by various shifting of its fixed loop order sub-graphs, see Table \ref{tbl:GC1}.
Another result is known for $\coH\left(\GC_1\right)$, dealing with \emph{extra differentials}, \cite[Section 4]{DGC1}.
The basic idea is to deform the standard differential $\delta$ to $\delta+\delta^{extra}$ making the complex (almost) acyclic.
The extra differential for $\GC_1$ is the Lie bracket with Maurer-Cartan element
$$
 m = 
\sum_{j\geq 1}
\frac{1}{(2j+1)!} \,
\begin{tikzpicture}[baseline=-.65ex,scale=.8]
 \node[int] (v) at (0,.5) {};
 \node[int] (w) at (0,-0.5) {};
 \draw (v) edge[very thick] node[right] {$\scriptstyle 2j+1$} (w);
\end{tikzpicture},
$$
where the thick line labeled by a number $2j+1$ represents a $2j+1$-fold edge, i.e., $2j+1$ edges connecting the same pair of vertices.

As the extra piece does not fix the loop order, the condition that extra piece is of degree 1 makes it suitable only for one odd parameter $n$, in this case $n=1$.

A spectral sequence on the loop order can be found such that the standard differential $\delta$ is the first differential. Therefore, on the first page of the spectral sequence we see the standard cohomology $\coH\left(\GC_1,\delta\right)$ we are interested in, and because the whole differential is acyclic, classes of it cancel with each other on further pages. We call this the \emph{cancelling mechanism}. We can say that classes come in pairs. Some cancellations of pairs, including some conjectural, are depicted as arrows in Table \ref{tbl:GC1}. It is conjectured in \cite[Conjecture 1]{DGC1} that loop classes are cancelled with elements of (shifted) algebra of 3-graphs.

\begin{table}[h]
$$
\begin{tikzpicture}[every node/.style={align=center,anchor=base,text width=2.15ex,nodes={fill=red!30}}]
\matrix (mag) [matrix of nodes,ampersand replacement=\&,row 1/.style={nodes={fill=white}},row 3/.style={nodes={fill=green!30}},column 1/.style={nodes={fill=white}}]
{
{}\& 0 \& 1 \& 2 \& 3 \& 4 \& 5 \& 6 \& 7 \& 8 \& 9 \&10 \&11 \&12 \&13 \&14 \&15 \&16 \&17 \&18 \&19 \&20 \&21 \\
-1\& 0 \& 0 \& 0 \& 0 \& 0 \& 0 \& 0 \& 0 \& 0 \& 0 \& 0 \& 0 \& 0 \& 0 \& 0 \& 0 \& 0 \& 0 \& 0 \& 0 \& 0 \& 0 \\
|[fill=white]| 0 \& 0 \& 0 \& 1 \& 0 \& 0 \& 0 \& 1 \& 0 \& 0 \& 0 \& 1 \& 0 \& 0 \& 0 \& 1 \& 0 \& 0 \& 0 \& 1 \& 0 \& 0 \& 0 \\
1 \& 0 \& |[fill=green!30]| 1 \& 0 \& 0 \& 0 \& 0 \& 0 \& 0 \& 0 \& 0 \& 0 \& 0 \& 0 \& 0 \& 0 \& 0 \& 0 \& 0 \& 0 \& 0 \& 0 \& 0 \\
2 \& 0 \& 0 \& 0 \& |[fill=green!30]| 1 \& 0 \& 0 \& 0 \& 0 \& 0 \& 0 \& 0 \& 0 \& 0 \& 0 \& 0 \& 0 \& 0 \& 0 \& 0 \& 0 \& 0 \& 0 \\
3 \& 0 \& 0 \& 0 \& 0 \& |[fill=white]| 0 \& |[fill=green!30]| 1 \& 0 \& 0 \& 0 \& 0 \& 0 \& 0 \& 0 \& 0 \& 0 \& 0 \& 0 \& 0 \& 0 \& 0 \& 0 \& 0 \\
4 \& 0 \& 0 \& 0 \& 0 \& 0 \& |[fill=white]| 0 \& |[fill=white]| 0 \& |[fill=green!30]| 2 \& 0 \& 0 \& 0 \& 0 \& 0 \& 0 \& 0 \& 0 \& 0 \& 0 \& 0 \& 0 \& 0 \& 0 \\
5 \& 0 \& 0 \& 0 \& 0 \& 0 \& 0 \& |[fill=white]| 1 \& |[fill=white]| 0 \& |[fill=white]| 0 \& |[fill=green!30]| 2 \& 0 \& 0 \& 0 \& 0 \& 0 \& 0 \& 0 \& 0 \& 0 \& 0 \& 0 \& 0 \\
6 \& 0 \& 0 \& 0 \& 0 \& 0 \& 0 \& 0 \& |[fill=white]| 0 \& |[fill=white]| 1 \& |[fill=white]| 0 \& |[fill=white]| 0 \& |[fill=green!30]| 3 \& 0 \& 0 \& 0 \& 0 \& 0 \& 0 \& 0 \& 0 \& 0 \& 0 \\
7 \& 0 \& 0 \& 0 \& 0 \& 0 \& 0 \& 0 \& 0 \& |[fill=white]| 0 \& |[fill=white]| 0 \& |[fill=white]| 2 \& |[fill=white]| 0 \& |[fill=white]| 0 \& |[fill=green!30]| 4 \& 0 \& 0 \& 0 \& 0 \& 0 \& 0 \& 0 \& 0 \\
8 \& 0 \& 0 \& 0 \& 0 \& 0 \& 0 \& 0 \& 0 \& 0 \& |[fill=white]| 0 \& |[fill=white]| ? \& |[fill=white]| ? \& |[fill=white]| ? \& |[fill=white]| ? \& |[fill=white]| ? \& |[fill=green!30]| 5 \& 0 \& 0 \& 0 \& 0 \& 0 \& 0 \\
9 \& 0 \& 0 \& 0 \& 0 \& 0 \& 0 \& 0 \& 0 \& 0 \& 0 \& |[fill=white]| ? \& |[fill=white]| ? \& |[fill=white]| ? \& |[fill=white]| ? \& |[fill=white]| ? \& |[fill=white]| ? \& |[fill=white]| ? \& |[fill=green!30]| 6 \& 0 \& 0 \& 0 \& 0 \\
10\& 0 \& 0 \& 0 \& 0 \& 0 \& 0 \& 0 \& 0 \& 0 \& 0 \& 0 \& |[fill=white]| ? \& |[fill=white]| ? \& |[fill=white]| ? \& |[fill=white]| ? \& |[fill=white]| ? \& |[fill=white]| ? \& |[fill=white]| ? \& |[fill=white]| ? \& |[fill=green!30]| 8 \& 0 \& 0 \\
11\& 0 \& 0 \& 0 \& 0 \& 0 \& 0 \& 0 \& 0 \& 0 \& 0 \& 0 \& 0 \& |[fill=white]| ? \& |[fill=white]| ? \& |[fill=white]| ? \& |[fill=white]| ? \& |[fill=white]| ? \& |[fill=white]| ? \& |[fill=white]| ? \& |[fill=white]| ? \& |[fill=white]| ? \& |[fill=green!30]| 9 \\
};

\draw (mag-3-4) edge[-latex, thick] (mag-5-5);
\draw (mag-6-7) edge[-latex, thick] (mag-8-8);
\draw (mag-3-8) edge[-latex, thick] (mag-7-9);
\draw (mag-7-9) edge[-latex, thick] (mag-9-10);
\draw (mag-8-11) edge[-latex, thick] node[above]{$?$} (mag-10-12);
\draw (mag-3-12) edge[-latex, thick] node[above]{$?$} (mag-9-13);
\draw (mag-3-16) edge[-latex, thick] node[above]{$?$} (mag-11-17);
\draw (mag-3-20) edge[-latex, thick] node[above]{$?$} (mag-13-21);

\draw (mag-2-1.north west) -- (mag-1-23.south east);
\draw (mag-1-2.north west) -- (mag-14-1.south east);
\end{tikzpicture}
$$
\caption{\label{tbl:GC1} The table of dimensions of $\coH(\GC_{1})$.
Note that rows are the same as in $\coH(\GC_{3})$ from Table \ref{tbl:GC3}, just with shifted degree.
Arrows represent cancellations using the extra differential $[m,\cdot]$, c.f.\ \cite[Table 2]{DGC1}. Question-marks on the arrows indicate that the cancellation is conjectural.
}
\end{table}

Further shifting leads to the next odd graph cohomology $\coH\left(\GC_{-1}\right)$ depicted in Table \ref{tbl:GC-1}. One can immediately see that classes are arranged in pairs as indicated by arrows, as if there is another extra differential suitable for the parameter $n=-1$ that cancels them.

A possible extra differential was proposed in \cite[Conjecture 2]{DGC3}. It is a differential $D$ that deletes a vertex, and reconnects its edges to other vertices, summed over all ways to attach them, and summed over all vertices to be deleted. Calculations on low degrees support the conjecture.

The purpose of this paper is to give an extra differential on $\coH\left(\GC_{-1}\right)$ that makes conjectural cancelling from Table \ref{tbl:GC-1} real. However, it is not $D$ from \cite[Conjecture 2]{DGC3}, and the conjecture remains, though of lesser importance since our result gives the cancelling we were looking for.

\begin{table}[h]
{\small $$
\begin{tikzpicture}[every node/.style={align=center,anchor=base,text width=2.15ex,nodes={fill=red!30}}]
\matrix (mag) [matrix of nodes,ampersand replacement=\&,row 1/.style={nodes={fill=white}},row 3/.style={nodes={fill=green!30}},column 1/.style={nodes={fill=white}}]
{
{}\& 0 \& 1 \& 2 \& 3 \& 4 \& 5 \& 6 \& 7 \& 8 \& 9 \&10 \&11 \&12 \&13 \&14 \&15 \&16 \&17 \&18 \&19 \&20 \&21 \&22 \&23 \&24 \&25 \&26 \&27 \\
-1\& 0 \& 0 \& 0 \& 0 \& 0 \& 0 \& 0 \& 0 \& 0 \& 0 \& 0 \& 0 \& 0 \& 0 \& 0 \& 0 \& 0 \& 0 \& 0 \& 0 \& 0 \& 0 \& 0 \& 0 \& 0 \& 0 \& 0 \& 0 \\
|[fill=white]| 0 \&  \&  \& 0 \& 0 \& 1 \& 0 \& 0 \& 0 \& 1 \& 0 \& 0 \& 0 \& 1 \& 0 \& 0 \& 0 \& 1 \& 0 \& 0 \& 0 \& 1 \& 0 \& 0 \& 0 \& 1 \& 0 \& 0 \& 0 \\
1 \&  \&  \&  \&  \& 0 \& |[fill=green!30]| 1 \& 0 \& 0 \& 0 \& 0 \& 0 \& 0 \& 0 \& 0 \& 0 \& 0 \& 0 \& 0 \& 0 \& 0 \& 0 \& 0 \& 0 \& 0 \& 0 \& 0 \& 0 \& 0 \\
2 \&  \&  \&  \&  \&  \&  \& 0 \& 0 \& 0 \& |[fill=green!30]| 1 \& 0 \& 0 \& 0 \& 0 \& 0 \& 0 \& 0 \& 0 \& 0 \& 0 \& 0 \& 0 \& 0 \& 0 \& 0 \& 0 \& 0 \& 0 \\
3 \&  \&  \&  \&  \&  \&  \&  \&  \& 0 \& 0 \& 0 \& 0 \& |[fill=white]| 0 \& |[fill=green!30]| 1 \& 0 \& 0 \& 0 \& 0 \& 0 \& 0 \& 0 \& 0 \& 0 \& 0 \& 0 \& 0 \& 0 \& 0 \\
4 \&  \&  \&  \&  \&  \&  \&  \&  \&  \&  \& 0 \& 0 \& 0 \& 0 \& 0 \& |[fill=white]| 0 \& |[fill=white]| 0 \& |[fill=green!30]| 2 \& 0 \& 0 \& 0 \& 0 \& 0 \& 0 \& 0 \& 0 \& 0 \& 0 \\
5 \&  \&  \&  \&  \&  \&  \&  \&  \&  \&  \&  \&  \& 0 \& 0 \& 0 \& 0 \& 0 \& 0 \& |[fill=white]| 1 \& |[fill=white]| 0 \& |[fill=white]| 0 \& |[fill=green!30]| 2 \& 0 \& 0 \& 0 \& 0 \& 0 \& 0 \\
6 \&  \&  \&  \&  \&  \&  \&  \&  \&  \&  \&  \&  \&  \&  \& 0 \& 0 \& 0 \& 0 \& 0 \& 0 \& 0 \& |[fill=white]| 0 \& |[fill=white]| 1 \& |[fill=white]| 0 \& |[fill=white]| 0 \& |[fill=green!30]| 3 \& 0 \& 0 \\
7 \&  \&  \&  \&  \&  \&  \&  \&  \&  \&  \&  \&  \&  \&  \&  \&  \& 0 \& 0 \& 0 \& 0 \& 0 \& 0 \& 0 \& 0 \& |[fill=white]| 0 \& |[fill=white]| 0 \& |[fill=white]| 2 \& |[fill=white]| 0 \\
8 \&  \&  \&  \&  \&  \&  \&  \&  \&  \&  \&  \&  \&  \&  \&  \&  \&  \&  \& 0 \& 0 \& 0 \& 0 \& 0 \& 0 \& 0 \& 0 \& 0 \& |[fill=white]| 0 \\
9 \&  \&  \&  \&  \&  \&  \&  \&  \&  \&  \&  \&  \&  \&  \&  \&  \&  \&  \&  \&  \& 0 \& 0 \& 0 \& 0 \& 0 \& 0 \& 0 \& 0 \\
10 \&  \&  \&  \&  \&  \&  \&  \&  \&  \&  \&  \&  \&  \&  \&  \&  \&  \&  \&  \&  \&  \& \& 0 \& 0 \& 0 \& 0 \& 0 \& 0 \\
11 \&  \&  \&  \&  \&  \&  \&  \&  \&  \&  \&  \&  \&  \&  \&  \&  \&  \&  \&  \&  \&  \&  \&  \& \& 0 \& 0 \& 0 \& 0 \\
};

\draw (mag-3-6.center) edge[-latex, thick] (mag-4-7);
\draw (mag-3-10.center) edge[-latex, thick] (mag-5-11);
\draw (mag-3-14.center) edge[-latex, thick] (mag-6-15);
\draw (mag-3-18.center) edge[-latex, thick] (mag-7-19);
\draw (mag-3-22.center) edge[-latex, thick] (mag-8-23);
\draw (mag-3-26.center) edge[-latex, thick] (mag-9-27);
\draw (mag-7-19.center) edge[-latex, thick] (mag-8-20);
\draw (mag-8-23.center) edge[-latex, thick] (mag-9-24);
\draw (mag-9-27.center) edge[-latex, thick] (mag-10-28);

\draw (mag-2-1.north west) -- (mag-1-29.south east);
\draw (mag-1-2.north west) -- (mag-14-1.south east);
\end{tikzpicture}
$$}
\caption{\label{tbl:GC-1} The table of dimensions of $\coH(\GC_{-1})$.
Conjectural cancelling of pending extra differential is indicated by arrows.
}
\end{table}

To do it we need to introduce \emph{hairy graph complexes} $\HGC_{m,n}$, where $m$ and $n$ range over integers. These complexes are spanned by graphs with external legs, also called ``hairs'' (see Section \ref{s:def} for the definition). At least one hair is required. Similarly to Kontsevich's graph complexes, there are essentially 4 hairy graph complexes that depend on the parities of $m$ and $n$. They compute the rational homotopy of the spaces of embeddings of disks modulo immersions, fixed at the boundary $ \overline{\mathrm{Emb}_\partial}(\mathbb{D}^m,\mathbb{D}^n)$, provided that $n-m\geq 3$, see \cite{AT,FTW}.
Some examples of hairy graphs are:
\begin{equation}
\label{eq:HGexample}
\begin{tikzpicture}[baseline=-1ex,scale=.7]
\node[int] (a) at (0,0) {};
\draw (a) edge (90:.3);
\draw (a) edge (120+90:.3);
\draw (a) edge (240+90:.3);
\end{tikzpicture}
\;,\quad
\begin{tikzpicture}[baseline=-2ex,scale=.7]
\node[int] (v1) at (-1,0){};
\node[int] (v2) at (0,.8){};
\node[int] (v3) at (1,0){};
\node[int] (v4) at (0,-.8){};
\draw (v1)  edge (v2);
\draw (v1)  edge (v4);
\draw (v2)  edge (v3);
\draw (v3)  edge (v4);
\draw (v4)  edge (v2);
\draw (v1)  edge (-1.3,0);
\draw (v3)  edge (1.3,0);
\end{tikzpicture}
\;,\quad
\begin{tikzpicture}[baseline=-1ex,scale=.7]
\node[int] (a) at (0,0) {};
\node[int] (b) at (90:1) {};
\node[int] (c) at (210:1) {};
\node[int] (d) at (330:1) {};
\draw (b) edge (0,1.3);
\draw (a) edge (b);
\draw (a) edge (c);
\draw (a) edge (d);
\draw (b) edge (c);
\draw (b) edge (d);
\draw (c) edge (d);
\end{tikzpicture}
\;.
\end{equation}

It is known \cite[Propositions 4.1 and 4.4]{vogel} that the first cohomology of the 2-hair subspace $\coH^{-1}(\mH^2\HGC_{1,3})$ and the first cohomology of the 3-hair subspace $\coH^1(\mH^3\HGC_{2,3})$ are each isomorphic to the commutative algebra of 3-graphs, i.e.\ third cohomology of the non-hairy graph complex $\coH^{-3}(\GC_3)$.

A couple of extra differentials is known for each of the four parity cases of hairy graph complexes. First, \cite{grt}, \cite{TW1} and \cite{TW2} introduce a deformed differentials on $\HGC_{n,n}$ and $\HGC_{n-1,n}$ such that there are quasi-isomorphisms
\begin{align*}
\GC_{n} &\to \left(\HGC_{n,n}, \delta+\chi\right), \\
\K \oplus \GC_{n} &\to \left(\HGC_{n-1,n}, \delta+[h_1,\cdot]\right).
\end{align*}
The map, and also the extra differential in the first case $\chi$, add a hair in all possible ways. A Maurer-Cartan element $h_1$ used for the extra differential in the second case is a particular sum of corollas.
In this paper we are dealing with $\GC_{-1}$, so we will use only the first result, see Subsection \ref{ss:chi}. Therefore, we are interested only in hairy graph complex $\HGC_{-1,-1}$.

The second extra differential on $\HGC_{m,n}$ for even $m$ is introduced in \cite{DGC2}. In that case, because of parities, it is possible to understand hairs as edges towards a special vertex. The extra differential is splitting that extra vertex.
For odd $n$ the second extra differential is proven in \cite{DGC3} to be $\Delta$ that connects a hair into an edge, see Subsection \ref{ss:Delta}.

The main idea of this paper is constructing
$$
\left(\mH_{\geq 0}\GC_{-1,-1},\delta+\Delta\right),
$$
where $\mH_{\geq 0}\GC_{-1,-1}$ is the hairy graph complex where graphs without hairs are also allowed.
Aforementioned results imply that it is quasi-isomorphic to $\GC_{-1}$.
The tempting extra differential on $\GC_{-1}$ will actually be the adding a hair differential $\chi$ on its quasi-isomorphic version $\left(\mH_{\geq 0}\GC_{-1,-1},\delta+\Delta\right)$.

Finally, we want to mention that having two extra differentials on the (essentially) same complex leads to the ``waterfall mechanism'' introduced for the first time in \cite{DGC2} for hairy graph complexes.
One starts from one cohomology class, finds its pair using one extra differential, then find its pair using another extra differential, then finds its pair using again the first extra differential, and so on. Weather this leads to an finite or infinite sequence of classes, or makes a loop, is an open question for each case. The picture of the ``waterfall mechanism'' for $\GC_{-1}$ is depicted in Table \ref{tbl:wfall}.

\subsection*{Structure of the paper}

In Section \ref{s:def} we define graph complexes needed in the paper. Subsections \ref{ss:Delta} and \ref{ss:chi} recalls two different extra differentials on hairy graph complex. In Subsection \ref{ss:main} we show our result.

\subsection*{Acknowledgements}

The author thanks Sergei Merkulov and Thomas Willwacher for the motivation, support and valuable discussion on the topic.

\section{Graph complexes}
\label{s:def}

In this section we quickly recall definitions of graph complexes needed in this paper that are all well known from the literature, and fix our notation.

Standard Kontsevich's graph complexes $\GC_n$ and hairy graph complexes $\HGC_{m,n}$ are in general defined e.g.\ in \cite{DGC3}.
In this paper we are only interested in the case when $m=n=-1$, that is $\GC_{-1}$ and $\HGC_{-1,-1}$. For simplicity, we will use the shorter notations $\GC:=\GC_{-1}$ and $\HGC:=\HGC_{-1,-1}$.

We work over a field $\K$ of characteristic zero. All vector spaces and differential graded vector spaces are assumed to be $\K$-vector spaces.

\subsection{From graphs to space of (co)invariants}
Consider the set $\bar\mV_v\bar\mE_e\bar\mH_s\grac$ containing directed graphs that:
\begin{itemize}
\item are connected;
\item have $v>0$ distinguishable vertices that are adjacent to at least 2 edges;
\item have $e\geq 0$ distinguishable directed edges;
\item have $s\geq 0$ distinguishable hairs attached to some vertices;
\end{itemize}
For some pictures of such graphs see \eqref{eq:HGexample}.

Let
\begin{equation}
\bar\mV_v\bar\mE_e\bar\mH_s\GC:=\langle\bar\mV_v\bar\mE_e\bar\mH_s\grac^{3}\rangle[-1+v-2e-s]
\end{equation}
be the vector space of formal series of $\bar\mV_v\bar\mE_e\bar\mH_s\grac$ with coefficients in $\K$. It is a graded vector space with a non-zero term only in degree $d=-1+v-2e-s$.

There is a natural right action of the group $\sym_v\times \sym_s\times \left(\sym_e\ltimes \sym_2^{\times e}\right)$ on $\bar\mV_v\bar\mE_e\bar\mH_s\grac$, where $S_v$ permutes vertices, $S_s$ permutes hairs, $S_e$ permutes edges and $S_2^{\times e}$ changes the direction of edges.
Let $\sgn_v$, $\sgn_s$, $\sgn_e$ and $\sgn_2$ be one-dimensional representations of $S_v$, respectively $S_s$, respectively $S_e$, respectively $S_2$, where the odd permutation reverses the sign. They can be considered as representations of the whole product $\sym_v\times \sym_s\times \left(\sym_e\ltimes \sym_2^{\times e}\right)$. Let us consider the space of invariants:
\begin{equation}
\mV_v\mE_e\mH_s\GC:=
\left(\bar\mV_v\bar\mE_e\bar\mH_s\GC\otimes\sgn_v\otimes\sgn_s\otimes\sgn_2^{\otimes e}\right)^{\sym_v\times \sym_s\times \left(\sym_e\ltimes \sym_2^{\times e}\right)}
\end{equation}
Because the group is finite, the space of invariants may be replaced by the space of coinvariants. The operation of taking (co)invarints effectively removes numbering on edges, and removes the edges directions and numberings of vertices and hairs up to sign.

\begin{rem}
Sign changes induced by reversing an edge direction and switching hairs imply that in $\mV_v\mE_e\mH_s\GC$ there are no tadpoles (edges that start and end at the same vertex), nor multiple hairs on the same vertex.
\end{rem}

\subsection{The differential}
\label{ss:dif}
The differential on $\mV_v\mE_e\mH_s\GC$ acts by splitting a vertex:
\begin{equation}
 \delta(\Gamma):=\sum_{x\in V(\Gamma)}\frac{1}{2}s_x(\Gamma),
\end{equation}
where $V(\Gamma)$ is the set of vertices of $\Gamma$, and $s_x$ stands for ``splitting of $x$'' and means inserting
\begin{tikzpicture}[scale=.5]
 \node[int] (a) at (0,0) {};
 \node[int] (b) at (1,0) {};
 \draw (a) edge (b);
\end{tikzpicture}
instead of the vertex $x$ and summing over all possible ways of connecting the edges that have been connected to $x$ to the new two vertices, requiring that at least one edge is connected to each of the new vertices.
One checks that the differential is well defined on the space of (co)invariants. Thus, for each $s\geq 0$ we get a graph complex
$$
\left(\mH_s\GC,\delta\right).
$$

\subsection{Full complexes}

Kontsevich's graph complex is
\begin{equation}
(\GC,\delta)=(\mH_0\GC,\delta),
\end{equation}
and hairy graph complex is
\begin{equation}
(\HGC,\delta)=\left(\prod_{s\geq 1}\mH_s\GC,\delta\right).
\end{equation}
We also need a combined complex
\begin{equation}
\left(\mH_{\geq 0}\GC,\delta\right)=\left(\prod_{s\geq 0}\mH_s\GC,\delta\right)=
(\GC,\delta)\oplus(\HGC,\delta).
\end{equation}

\begin{rem}
\label{rem:val}
Recall that each vertex in our graphs has to meet at least 2 edges. No tadpoles are allowed, so that means that the valence of each vertex of a graph in the complex $\GC=\GC_{-1}$ is at least 2. That complex is often called $\GC^2_{-1}$ in the literature.

Similarly, for hairy graphs in our complex $\HGC=\HGC_{-1,-1}$, each non-hairy vertex has to be at least 2-valent, and each hairy vertex has to be at least 3-valent, since there can not be more than one hair on a vertex. This valence convention is not so standard, but it is considered in \cite{DGC3}, where our complex $\HGC$ is called $\mH^{\geq 1}\fHGCc_{-1,-1}^\natural$.

It is well known and easy to check that these valence conditions do not change the cohomology essentially, c.f.\ \cite[Proposition 3.4]{grt}.
\end{rem}

\section{Extra differentials}

\subsection{Connecting a hair}
\label{ss:Delta}
In the project of finding extra differentials on graph complexes, an extra differential $\Delta$ on $\HGC_{m,n}$ for odd $m$ was conjectured in \cite{DGC2}:
\begin{equation}
\Delta\Gamma=\sum_{t\in H(\Gamma)}\Delta_t\Gamma,
\end{equation}
where $H(\Gamma)$ is the set of hairs in $\Gamma$ and $\Delta_t\Gamma$ is the graph obtained from $\Gamma$ by transforming hair $t$ on a vertex $x$ to an edge from $x$ to all other vertices, summed over all vertices. One checks that $\Delta$ is well defined on the space of (co)invariants. We have \cite[Lemma 1]{DGC2}:
\begin{equation}
\Delta^2=0,
\end{equation}
\begin{equation}
(\delta+\Delta)=0.
\end{equation}
On hairy complexes $\HGC_{-1,n}$, including our $\HGC=\HGC_{-1,-1}$, $\Delta$ is of degree 1, so upper equalities imply that $\Delta$ and $(\delta+\Delta)$ are differentials.

Extra differential $\Delta$ reduces number or hairs by 1, and it is naturally defined on $\mH_{\geq 0}\GC$, forming the complex
$$
\left(\mH_{\geq 0}\GC,\Delta\right).
$$
However, we will often need restricted and projected differential $\Delta:\HGC\rightarrow\HGC$ that forms a complex
$$
(\HGC,\Delta).
$$
Here, if a graph $\Gamma$ has only one hair, $\Delta\Gamma$ is considered to be zero.
In both upper complexes we may put differential $\delta+\Delta$ instead of $\Delta$.

The conjecture from \cite{DGC2} was proven in \cite[Theorem 1.1]{DGC3}. However, the convention on valences in this paper is slightly different from the one there, as discussed in Remark \ref{rem:val}, so we technically need an equivalent result given in \cite[Proposition 6.10]{DGC3}:

\begin{thm}[{\cite[Proposition 6.10]{DGC3}\footnote{Strictly speaking, \cite{DGC3} deals with the complex $\HGC_{-1,1}$. It is easy to see that the second index is irrelevant for the differential $\delta+\Delta$, so we can use that result for our $\HGC=\HGC_{-1,-1}$.}}]
\label{thm:DGC3}
The complex
$$
\left(\HGC,\delta+\Delta\right)
$$
is acyclic.
\end{thm}

\begin{cor}
\label{cor:qi}
The inclusion
$$
\left(\GC,\delta\right)\hookrightarrow
\left(\mH_{\geq 0}\GC,\delta+\Delta\right)
$$
is a quasi-isomorphism.
\end{cor}
\begin{proof}
We have
$$
\mH_{\geq 0}\GC=\HGC\oplus \GC.
$$
On the mapping cone of the inclusion we make a simple spectral sequence of two rows: hairy and non-hairy part. On the first page in the first (hairy) row there is the complex $\left(\HGC,\delta+\Delta\right)$, being acyclic by Theorem \ref{thm:DGC3}. In the second (non-hairy) row there is a mapping cone of the identity $\left(\GC,\delta\right)\rightarrow\left(\GC,\delta\right)$, again acyclic.
The spectral sequence clearly converges correctly. This leads to the result.
\end{proof}

\subsection{Adding a hair}
\label{ss:chi}

An easy extra differential on $\HGC_{n,n}$ is adding a hair:
\begin{equation}
\chi\Gamma=\sum_{x\in V(\Gamma)}\chi_x\Gamma,
\end{equation}
where $\chi_x$ adds a hair on vertex $x$. It is straightforward to check that
\begin{equation}
\chi^2=0,
\end{equation}
\begin{equation}
(\delta+\chi)=0.
\end{equation}
On hairy complexes $\HGC_{n,n}$, including our $\HGC=\HGC_{-1,-1}$, $\chi$ is of degree 1, so upper equalities imply that $\chi$ and $(\delta+\chi)$ are differentials.

Recall that \cite[Theorem 1]{TW2} implies that
$$
(\GC_{n},\delta) \tto{\chi} \left(\HGC_{n,n}, \delta+\chi\right)
$$
is a quasi-isomorphism. Equivalently, its mapping cone
$$
\left(\mH_{\geq 0}\GC_{n,n},\delta+\chi\right)
$$
is acyclic.

The result from \cite[Theorem 1]{TW2} is much stronger than the upper insertions\footnote{It states also that $(\HGC_{n,n},\delta)$ can be split into two sub-complexes such that $\chi$ sends one sub-complex into another.}, and we will not need its full strength here.
Therefore, for completeness, and in order not to deal with technicalities of different valence conditions (c.f.\ Remark \ref{rem:val}), we write a separate simple proof here.

\begin{prop}
\label{prop:chi}
The complex
$$
\left(\mH_{\geq 0}\GC,\chi\right)
$$
is acyclic.
\end{prop}
\begin{proof}
Differential $\chi$ adds a hair to a vertex, summed over all (non-hairy) vertices. Let $\beta:\mH_{\geq 0}\GC\rightarrow\mH_{\geq 0}\GC$ be the sum over all hairs of deleting that hair. One easily sees that
$$
(\chi\beta+\beta\chi)\Gamma=v\Gamma
$$
where $v$ is the number of vertices in $\Gamma$. This implies the result.
\end{proof}

\begin{cor}
\label{cor:delta+chi}
The complex
$$
\left(\mH_{\geq 0}\GC,\delta+\chi\right)
$$
is acyclic.
\end{cor}
\begin{proof}
On the complex we make the spectral sequence on the number of edges $e$. On the first page there is the cohomology of $\left(\mH_{\geq 0}\GC,\chi\right)$, being acyclic by Proposition \ref{prop:chi}.

For a graph to be connected it must hold that $e\geq v-1$. Therefore, in the fixed degree $d$ we have
$$
d=1-v+2e+h\geq e+h.
$$
In particular $d\geq e$, $d\geq h$ and $d+1\geq v$. So, in the fixed degree, number of vertices, edges and hairs are all bounded, and therefore there are only finitely many possible graphs. So, the complex is finitely dimensional in each degree and standard spectral sequence arguments (cf.\ \cite[Chapter 5]{Wei} or \cite[Appendix C]{DGC1}) imply the correct convergence. Hence the result follows.
\end{proof}

\subsection{Combining two extra differentials}
\label{ss:main}
One easily checks that connecting a hair $\Delta$ and adding a hair $\chi$ anti-commute, so $\Delta+\chi$ and $\delta+\Delta+\chi$ are also differentials on $\mH_{\geq 0}\GC$.

\begin{cor}
The complex
$$
\left(\mH_{\geq 0}\GC,\delta+\Delta+\chi\right)
$$
is acyclic.
\end{cor}
\begin{proof}
The proof is exactly the same as the one of Corollary \ref{cor:delta+chi}.
\end{proof}

\begin{thm}
\label{thm:main}
There is a spectral sequence converting to
$$
H\left(\mH_{\geq 0}\GC,\delta+\Delta+\chi\right)=0
$$
whose term on the second page is
$$
H\left(\mH_{\geq 0}\GC,\delta+\Delta\right)=
H\left(\GC,\delta\right).
$$
\end{thm}
\begin{proof}
The spectral sequence is on number $e-v-h$. It converges correctly because of the same reason as in the proof of Corollary \ref{cor:delta+chi}.
On the first page there is cohomology of $\left(\mH_{\geq 0}\GC,\delta+\Delta\right)$, being equal to cohomology of $\left(\GC,\delta\right)$ by Corollary \ref{cor:qi}.
\end{proof}

Table \ref{tbl:GC-1} represents the second page of the spectral sequence from Theorem \ref{thm:main}. With the corollary we have proven that there is indeed cancelling as depicted by the arrows.
In Table \ref{tbl:wfall} we repeat the same cancellations of cohomological classes of shifted complex $\left(\GC_{1},\delta\right)$, together with already known cancellations on that complex from \cite[Section 4]{DGC1} depicted in Table \ref{tbl:GC1}.

\begin{table}[h]
$$
\begin{tikzpicture}[every node/.style={align=center,anchor=base,text width=2.15ex,nodes={fill=red!30}}]
\matrix (mag) [matrix of nodes,ampersand replacement=\&,row 1/.style={nodes={fill=white}},row 3/.style={nodes={fill=green!30}},column 1/.style={nodes={fill=white}}]
{
{}\& 0 \& 1 \& 2 \& 3 \& 4 \& 5 \& 6 \& 7 \& 8 \& 9 \&10 \&11 \&12 \&13 \&14 \&15 \&16 \&17 \&18 \&19 \&20 \&21 \&22\\
-1\& 0 \& 0 \& 0 \& 0 \& 0 \& 0 \& 0 \& 0 \& 0 \& 0 \& 0 \& 0 \& 0 \& 0 \& 0 \& 0 \& 0 \& 0 \& 0 \& 0 \& 0 \& 0 \& 0 \\
|[fill=white]| 0 \& 0 \& 0 \& 1 \& 0 \& 0 \& 0 \& 1 \& 0 \& 0 \& 0 \& 1 \& 0 \& 0 \& 0 \& 1 \& 0 \& 0 \& 0 \& 1 \& 0 \& 0 \& 0 \& 1 \\
1 \& 0 \& |[fill=green!30]| 1 \& 0 \& 0 \& 0 \& 0 \& 0 \& 0 \& 0 \& 0 \& 0 \& 0 \& 0 \& 0 \& 0 \& 0 \& 0 \& 0 \& 0 \& 0 \& 0 \& 0 \& 0 \\
2 \& 0 \& 0 \& 0 \& |[fill=green!30]| 1 \& 0 \& 0 \& 0 \& 0 \& 0 \& 0 \& 0 \& 0 \& 0 \& 0 \& 0 \& 0 \& 0 \& 0 \& 0 \& 0 \& 0 \& 0 \& 0 \\
3 \& 0 \& 0 \& 0 \& 0 \& |[fill=white]| 0 \& |[fill=green!30]| 1 \& 0 \& 0 \& 0 \& 0 \& 0 \& 0 \& 0 \& 0 \& 0 \& 0 \& 0 \& 0 \& 0 \& 0 \& 0 \& 0 \& 0 \\
4 \& 0 \& 0 \& 0 \& 0 \& 0 \& |[fill=white]| 0 \& |[fill=white]| 0 \& |[fill=green!30]| 2 \& 0 \& 0 \& 0 \& 0 \& 0 \& 0 \& 0 \& 0 \& 0 \& 0 \& 0 \& 0 \& 0 \& 0 \& 0 \\
5 \& 0 \& 0 \& 0 \& 0 \& 0 \& 0 \& |[fill=white]| 1 \& |[fill=white]| 0 \& |[fill=white]| 0 \& |[fill=green!30]| 2 \& 0 \& 0 \& 0 \& 0 \& 0 \& 0 \& 0 \& 0 \& 0 \& 0 \& 0 \& 0 \& 0 \\
6 \& 0 \& 0 \& 0 \& 0 \& 0 \& 0 \& 0 \& |[fill=white]| 0 \& |[fill=white]| 1 \& |[fill=white]| 0 \& |[fill=white]| 0 \& |[fill=green!30]| 3 \& 0 \& 0 \& 0 \& 0 \& 0 \& 0 \& 0 \& 0 \& 0 \& 0 \& 0 \\
7 \& 0 \& 0 \& 0 \& 0 \& 0 \& 0 \& 0 \& 0 \& |[fill=white]| 0 \& |[fill=white]| 0 \& |[fill=white]| 2 \& |[fill=white]| 0 \& |[fill=white]| 0 \& |[fill=green!30]| 4 \& 0 \& 0 \& 0 \& 0 \& 0 \& 0 \& 0 \& 0 \& 0 \\
8 \& 0 \& 0 \& 0 \& 0 \& 0 \& 0 \& 0 \& 0 \& 0 \& |[fill=white]| 0 \& |[fill=white]| ? \& |[fill=white]| ? \& |[fill=white]| ? \& |[fill=white]| ? \& |[fill=white]| ? \& |[fill=green!30]| 5 \& 0 \& 0 \& 0 \& 0 \& 0 \& 0 \& 0 \\
9 \& 0 \& 0 \& 0 \& 0 \& 0 \& 0 \& 0 \& 0 \& 0 \& 0 \& |[fill=white]| ? \& |[fill=white]| ? \& |[fill=white]| ? \& |[fill=white]| ? \& |[fill=white]| ? \& |[fill=white]| ? \& |[fill=white]| ? \& |[fill=green!30]| 6 \& 0 \& 0 \& 0 \& 0 \& 0 \\
10\& 0 \& 0 \& 0 \& 0 \& 0 \& 0 \& 0 \& 0 \& 0 \& 0 \& 0 \& |[fill=white]| ? \& |[fill=white]| ? \& |[fill=white]| ? \& |[fill=white]| ? \& |[fill=white]| ? \& |[fill=white]| ? \& |[fill=white]| ? \& |[fill=white]| ? \& |[fill=green!30]| 8 \& 0 \& 0 \& 0 \\
11\& 0 \& 0 \& 0 \& 0 \& 0 \& 0 \& 0 \& 0 \& 0 \& 0 \& 0 \& 0 \& |[fill=white]| ? \& |[fill=white]| ? \& |[fill=white]| ? \& |[fill=white]| ? \& |[fill=white]| ? \& |[fill=white]| ? \& |[fill=white]| ? \& |[fill=white]| ? \& |[fill=white]| ? \& |[fill=green!30]| 9 \& 0 \\
};

\draw (mag-2-1.north west) -- (mag-1-24.south east);
\draw (mag-1-2.north west) -- (mag-14-1.south east);

\draw (mag-3-4) edge[-latex, thick,dotted] (mag-5-5);
\draw (mag-6-7) edge[-latex, thick,dotted] (mag-8-8);
\draw (mag-3-8) edge[-latex, thick,dotted] (mag-7-9);
\draw (mag-7-9) edge[-latex, thick,dotted] (mag-9-10);
\draw (mag-8-11) edge[-latex, thick,dotted] node[above]{$?$} (mag-10-12);
\draw (mag-3-12) edge[-latex, thick,dotted] node[above]{$?$} (mag-9-13);
\draw (mag-3-16) edge[-latex, thick,dotted] node[above]{$?$} (mag-11-17);
\draw (mag-3-20) edge[-latex, thick,dotted] node[above]{$?$} (mag-13-21);

\draw[-latex, thick] (mag-3-4) edge (mag-4-3.center);
\draw[-latex, thick] (mag-3-8) edge (mag-5-5.center);
\draw[-latex, thick] (mag-3-12) edge (mag-6-7.center);
\draw[-latex, thick] (mag-3-16) edge (mag-7-9.center);
\draw[-latex, thick] (mag-3-20) edge (mag-8-11.center);
\draw[-latex, thick] (mag-3-24) edge (mag-9-13.center);
\draw[-latex, thick] (mag-7-9) edge (mag-8-8.center);
\draw[-latex, thick] (mag-8-11) edge (mag-9-10.center);
\draw[-latex, thick] (mag-9-13) edge (mag-10-12.center);
\end{tikzpicture}
$$
\caption{\label{tbl:wfall} The table of $H\left(\GC_{1},\delta\right)$.
Solid arrows represent cancellations induced by Theorem \ref{thm:main}, while dotted arrows represent cancellations from \cite[Section 4]{DGC1}. Conjectural cancellations are indicated with question mark.}
\end{table}

\appendix

\section{An example of explicit cancelling}

In the appendix we will write explicit way of the first cancelling, where the loop class generated by
$$
L_3:=
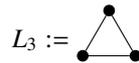
\begin{tikzpicture}[baseline=-.5ex,scale=.4]
\node[int] (b) at (90:1) {};
\node[int] (c) at (210:1) {};
\node[int] (d) at (330:1) {};
\draw (b) edge (c);
\draw (b) edge (d);
\draw (c) edge (d);
\end{tikzpicture}
$$
is cancelled with Theta class generated by
$$
\Theta:=
\begin{tikzpicture}[baseline=-.5ex,scale=.7]
\node[int] (a) at (0,0) {};
\node[int] (b) at (1,0) {};
\draw (a) edge (b);
\draw (a) edge[bend left=40] (b);
\draw (a) edge[bend right=40] (b);
\end{tikzpicture}\,.
$$

The graph $L_3\in\GC\subset\mH_{\geq 0}\GC$ generates a class in $\coH\left(\mH_{\geq 0}\GC,\delta+\Delta\right)$. On the second page of the spectral sequence of Theorem \ref{thm:main} that graph is mapped to
$$
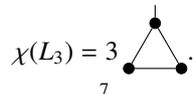
\chi(L_3)=3\,
\begin{tikzpicture}[baseline=-.5ex,scale=.4]
\node[int] (b) at (90:1) {};
\node[int] (c) at (210:1) {};
\node[int] (d) at (330:1) {};
\draw (b) edge (0,1.6);
\draw (b) edge (c);
\draw (b) edge (d);
\draw (c) edge (d);
\end{tikzpicture}.
$$
The spectral sequence ends here, so $\chi(L_3)$ represents a class in $\coH\left(\mH_{\geq 0}\GC,\delta+\Delta\right)$. However, $\chi(L_3)$ has a hair, so it is not in the image of the inclusion $(\GC,\delta)\hookrightarrow\left(\mH_{\geq 0}\GC,\delta+\Delta\right)$.
To find what is a representative of this class in $\coH(\GC,\delta)$, we need to find another representative of the class $[\chi(L_3)]$ that is hairless. Indeed, we have
$$
(\delta+\Delta)\;
\begin{tikzpicture}[baseline=-.5ex,scale=.7]
\node[int] (a) at (0,0) {};
\node[int] (b) at (1,0) {};
\draw (b) edge (1.3,0);
\draw (a) edge[bend left=40] (b);
\draw (a) edge[bend right=40] (b);
\end{tikzpicture}
=
\begin{tikzpicture}[baseline=-.5ex,scale=.4]
\node[int] (b) at (90:1) {};
\node[int] (c) at (210:1) {};
\node[int] (d) at (330:1) {};
\draw (b) edge (0,1.6);
\draw (b) edge (c);
\draw (b) edge (d);
\draw (c) edge (d);
\end{tikzpicture}
+
\begin{tikzpicture}[baseline=-.5ex,scale=.7]
\node[int] (a) at (0,0) {};
\node[int] (b) at (1,0) {};
\draw (a) edge (b);
\draw (a) edge[bend left=40] (b);
\draw (a) edge[bend right=40] (b);
\end{tikzpicture}\,,
$$
so $\chi(L_3)\sim -3\Theta$.
Therefore, the class $[L_3]$ cancels the class $-3[\Theta]$.

\end{document}